\documentclass{amsart}
\usepackage{amssymb, amsmath, latexsym, mathabx, dsfont, tikz}

\renewcommand{\baselinestretch}{\baselinestretch}
\renewcommand{\baselinestretch}{1.1}
\numberwithin{equation}{section}

\newtheorem{thm}{Theorem}[section]
\newtheorem{lem}[thm]{Lemma}
\newtheorem{cor}[thm]{Corollary}

\newtheorem{exam}[thm]{Example}

\theoremstyle{definition}
\newtheorem{defn}[thm]{Definition}

\numberwithin{equation}{section}

\newcommand{\ra}{{\rightarrow}}

\newcommand{\gen}{\text{gen}}
\newcommand{\spn}{\text{spn}}
\newcommand{\ord}{\text{ord}}

\newcommand{\z}{{\mathbb Z}}
\newcommand{\q}{{\mathbb Q}}

\newcommand{\e}{{\epsilon}}

\newcommand{\gr}{{\mathfrak{G}_{L,p}(m)}}

\begin{document}

\title{genus-correspondences respecting spinor genus}

\author{Jangwon Ju and Byeong-Kweon Oh }

\address{Department of Mathematical Sciences, Seoul National University, Seoul 08826, Korea}
\email{jjw@snu.ac.kr}

\address{Department of Mathematical Sciences and Research Institute of Mathematics, Seoul National University, Seoul 08826, Korea}
\email{bkoh@snu.ac.kr}
\thanks{This work of the second author was supported by the National Research Foundation of Korea (NRF-2014R1A1A2056296).}

\subjclass[2000]{Primary 11E12, 11E20}

\keywords{genus-correspondence, spinor genus}

\begin{abstract} For two positive definite integral ternary quadratic forms $f$ and $g$ and a positive integer $n$,  if $n\cdot g$ is represented by $f$ and $n\cdot dg=df$, then the pair $(f,g)$ is called a {\it representable pair by scaling $n$}. The set of all representable pairs in $\gen(f)\times \gen(g)$ is called a genus-correspondence. In \cite {Jagy}, Jagy conjectured that  if $n$ is square free and the number of spinor genera in the genus of $f$ equals to the number of spinor genera in the genus of $g$, 
then such a genus-correspondence respects spinor genus in the sense that  for  any representable pairs 
$(f,g), (f',g')$ by scaling $n$,  $f' \in \spn(f)$ if and only if   $g' \in \spn(g)$.
In this article, we show that by giving a counter example, Jagy's conjecture does not hold. Furthermore,  we provide a necessary and sufficient condition for a genus-correspondence to respect spinor genus.
\end{abstract}

\maketitle \pagestyle{myheadings}
 \markboth {Jangwon Ju and Byeong-Kweon Oh}{genus-correspondences respecting spinor genus}

\maketitle

\section{Introduction} 

For a positive definite integral ternary quadratic form 
$$
f(x,y,z)=ax^2+by^2+cz^2+pyz+qzx+rxy \  \ (a,b,c,p,q,r \in \z),
$$ 
it is quite an old problem determining the set $Q(f)$ of all positive integers $k$ such that $f(x,y,z)=k$ has an integer solution.  If the class number of $f$ is one, then one may easily compute the set $Q(f)$ by using, so called, the local-global principle. However, if the class number of $f$ is bigger than $1$,  then  determining the set $Q(f)$ exactly seems to be quite difficult, except some very special ternary quadratic forms. If the integer $k$ is sufficiently large, then   
the theorem of Duke and Schulze-Pillot  in \cite{ds} implies that if $k$ is primitively represented by the spinor genus of $f$, then $k$ is represented by $f$ itself.     

Recently,  W. Jagy proved in \cite{Jagy} that for any square free integer $k$ that is represented by a sum of two integral squares, it is represented by any ternary quadratic form in the spinor genus $x^2+y^2+16kz^2$. To prove this, he introduced, so called {\it a genus-correspondence}, and proved some interesting properties on the genus-correspondence. To be more precise, 
let $\gen(f)$ ($\spn(f)$) be the set of genus (spinor genus, respectively) of $f$, for any ternary quadratic form $f$. 
Let $f$ and $g$ be positive definite integral ternary quadratic forms, and assume that there is a positive integer $n$ such that 
\begin{equation}\label{gen-co}
\text{$n\cdot g$ is represented by $f$} \qquad \text{and} \qquad  n\cdot dg=df.
\end{equation}
  In this article, we denote such a pair $(f,g)$ by a {\it representable pair by scaling $n$}.  Note that $n\cdot f$ is also represented by $g$ for any representable pair $(f,g)$ by scaling $n$.
   As stated in \cite {Jagy}, W. K. Chan proved that for any ternary quadratic form $f'\in \gen(f)$,  there is a ternary quadratic form $g'\in \gen(g)$ such that $(f',g')$ is a representable pair by scaling $n$, and conversely  for any $\tilde g \in \gen(g)$, there is an $\tilde f \in \gen(f)$ such that $(\tilde f,\tilde g)$ is also a representable pair by scaling $n$.  Jagy defined the set of representable pairs by scaling $n$  by a {\it genus-correspondence} and proved some properities on a genus-correspondence. He also conjectured that if $n$ is square free and the number of spinor genera in the genus of $f$ equals to the number of spinor genera in the genus of $g$, 
then such a genus-correspondence {\it respects spinor genus} in the sense that  for  any representable pairs 
$(f,g), (f',g')$ by scaling $n$, where  $f' \in \gen(f)$ and $g' \in \gen(g)$,
\begin{equation} \label{respect}
f' \in \spn(f) \qquad \text{if and only if} \qquad  g' \in \spn(g).
\end{equation}

In this article, we give an  example such that Jagy's conjecture does not hold. In fact, the concept of ``genus-correspondence" in \cite{Jagy} is a little bit ambiguous. We modify the notion of a genus-correspondence as follows:  For a positive integer $n$, let $\mathfrak C$ be a subset of $\gen(f)\times \gen(g)$ such that each pair in $\mathfrak C$ is a representable pair by scaling $n$.  We say $\mathfrak C$ is a genus-correspondence if  for any $f' \in \gen(f)$, there is an $g' \in \gen(g)$ such that $(f',g') \in \mathfrak C$, and conversely,  for any $\tilde g\in \gen(g)$, there is an $\tilde f \in \gen(f)$ such that $(\tilde f,\tilde g) \in \mathfrak C$. Note that the set of all representable pairs by scaling $n$ is a genus-correspondence.  We show that without assumption that $n$ is square free, there is a genus-correspondence respecting spinor genus if the number of spinor genera in $\gen(f)$ is equal to the number of spinor genera in $\gen(g)$. 

In Section 5, we discuss when Jagy's original conjecture is true. We provide a necessary and sufficient condition for the genus-correspondence consisting of all representable pairs by scaling $n$ in $\gen(f)\times \gen(g)$ to respect spinor genus under the assumption that $n$ is square free and  the number of spinor genera in $\gen(f)$ is equal to the number of spinor genera in $\gen(g)$. 
 
The subsequent discussion will be conducted in the better adapted geometric language of quadratic spaces and lattices. The term {\it lattice} will always refer to a {\it non-classic integral} $\z$-lattice on an $n$-dimensional positive definite quadratic space $\q$.  Here a $\z$-lattice $L$ is said to be non-classic integral if the norm ideal $\mathfrak{n}(L)$ of $L$ is contained in $\z$. 
 The discriminant of a lattice $L$ is denoted by $dL$ and the number of (proper) spinor genera in $\gen (L)$ is denoted by $g(L)$. For any  rational number $a$, $L^a$ is the lattice whose bilinear map $B$ is scaled by $a$. 
 
 Let $L= \z x_1 + \z x_2 + \cdots + \z x_k$ be a $\z$-lattice of rank $k$. We write
$$
L \simeq (B(x_i,x_j)).
$$
The right hand side matrix is called a matrix presentation of $L$. If $L$ admits an orthogonal basis $\{x_1,x_2,\dots,x_k\}$, then we simply write
$$
L \simeq \langle Q(x_1),Q(x_2),\dots,Q(x_k) \rangle.
$$
Throughout this paper, we say a $\z$-lattice $L$ is {\it primitive} if the norm ideal $\mathfrak{n}(L)$ is exactly $\z$. For a prime $p$, the group of units in $\z_p$ is denoted by $\z_p^{\times}$. Unless confusion arises, we will simply use $\Delta_p$ to denote a non-square element in $\z_p^{\times}$, when $p$ is odd.

We denote by $\langle a,b,c,s,t,u \rangle$ for the ternary $\z$-lattice with a matrix presentation 
$$
\begin{pmatrix}a&u&t\\u&b&s\\t&s&c\end{pmatrix},
$$
for convenience. For any $\z$-lattice $L$, the equivalence class containing $L$ up to isometry is denoted by $[L]$. 
 For any integer $a$, we say that $\frac a2$ is divisible by a prime $p$ if $p$ is odd and $a \equiv 0~(\text{mod}~p)$, or $p=2$ and $a\equiv0~(\text{mod}~4)$. 

Any unexplained notations and terminologies can be found in \cite{ki} or \cite{om}. 

\section{Watson's transformations on the set of spinor genera}

Let $L$ be a non-classic integral $\z$-lattice on a quadratic space $V$. For a prime $p$, we define
$$
\Lambda_p(L)= \{ x \in L : Q(x + z) \equiv Q(z) \  (\text{mod} \ p) \mbox{ for all $z \in L$}\}.
$$
Let $\lambda_p(L)$ be the primitive lattice obtained from $\Lambda_p(L)$ by scaling $V=L\otimes \mathbb Q$ by a suitable rational number. Recall that a $\z$-lattice $L$ is called primitive if $\mathfrak n(L)=\z$.
 For general properties of $\Lambda_p$-transformation, see \cite{lambda} and \cite{lambda2}.

For $L' \in \gen (L)$ ($L' \in \spn (L)$) and any prime $p$, one may easily show that $\lambda_p(L') \in \gen(\lambda_p(L))$ ($\lambda_p(L') \in \spn (\lambda_p(L))$, respectively). It is well known that as a map,
\begin{equation}\label{lambda5}
\lambda_p:\gen(L)\longrightarrow\gen(\lambda_p(L))
\end{equation}
is surjective. Furthermore, $\lambda_p(\spn(L))=\spn(\lambda_p(L))$. If we define $\gen(L)_S$ the set of all spinor genera in $\gen(L)$, then the map 
$$
\lambda_p : \gen(L)_S \longrightarrow \gen(\lambda_p(L))_S
$$ 
given by $\spn(L')\mapsto \spn(\lambda_p(L'))$ for any $\spn(L')\in\gen(L)_S$ is well-defined and surjective. In particular, $g(L)\geq g(\lambda_p(L))$ for any prime $p$.

Henceforth, $L$  is always a positive definite non-classic integral ternary $\z$-lattice. 

\begin{defn}\label{H-type defn}
For a $\z$-lattice $L$ and a prime $p$, if $g(L)=g(\lambda_p(L))$, then we say the lattice $L$ is of $H$-type at $p$.
\end{defn}
From the definition, if $L$ is of $H$-type at $p$, then so is $L'$ for any $L' \in \gen(L)$.
\begin{lem}\label{odd}
Let $L$ be a primitive ternary $\z$-lattice and let $p$ be an odd prime. Assume that after scaling by a unit in $\z_p$ suitably, \
$$
L_p \simeq \langle1,p^\alpha\e_1,p^\beta\e_2\rangle,
$$
where $\alpha,\beta (\alpha\leq\beta)$ are nonnegative integers and $\e_1,\e_2 \in \{1,\Delta_p\}$. If $L$ is not of $H$-type at $p$, then the pairs $(\alpha,\beta)$, $(\e_1,\e_2)$ satisfy one of the conditions in Table 1.
\end{lem}

\begin{center}
\begin{footnotesize}
\renewcommand{\arraystretch}{1.5}
\begin{tabular}{ll}\multicolumn{2}{l}{\hspace{0.5cm}Table 1 (odd case)}\\
\hline
$(\alpha, \beta)$ & $(\epsilon_1, \epsilon_2)$\\ \hline
$(1,2)$&$(1,1)$\\
$(1,2)$&$(\Delta,1)$\\
$(2,k),~(k \geq 3)$&$(1,1)$\\
$(2,2k+1),~(k\geq1)$&$(1,\Delta)$\\ \hline
\end{tabular}
\end{footnotesize}
\end{center}

\begin{proof}
By 102:7 of \cite{om}, we know that 
$$
g(L)=[J_{\q} : P_DJ_{\q}^L] \qquad \text{and} \qquad g(\lambda_p(L)) = [J_{\q} : P_DJ_{\q}^{\lambda_p(L)}],
$$
where $D$ is the set of positive rational numbers.
Clearly, $\theta(O^+(\lambda_p(L)_q))=\theta(O^+(L_q))$ for any prime $q\neq p$.
Now one may easily check that if the pairs $(\alpha,\beta)$, $(\e_1,\e_2)$  do not satisfy one of the conditions in Table 1, then $\theta(O^+(\lambda_p(L)_p))=\theta(O^+(L_p))$. This implies that $g(L)=g(\lambda_p(L))$. 
\end{proof}

\begin{lem}\label{even}
Let $L$ be a primitive ternary $\z$-lattice.   If $L$ is not of $H$-type at $2$, then there is an $\eta \in \z_2^{\times}$ such that
$$
(L^{\eta})_2\simeq\langle1,2^\alpha\e_1,2^\beta\e_2\rangle,
$$
where  $\alpha, \beta (\alpha \leq \beta)$ are nonnegative integers and $\e_1,\e_2\in\z_2^{\times}$,  and the pairs $(\alpha,\beta)$, $(\e_1,\e_2)$ satisfy one of the conditions in Table 2. 
\end{lem}

\begin{proof}
The proof is quite similar to the above lemma. For the computation of the spinor norm map, see \cite{local}. 
\end{proof}

\begin{center}
\begin{footnotesize}
\renewcommand{\arraystretch}{1.5}
\begin{tabular}{ll||ll}\multicolumn{4}{l}{\hspace{4cm}Table 2 (Even case)}\\
\hline$(\alpha, \beta)$ & $(\epsilon_1, \epsilon_2)$ & $(\alpha, \beta)$ & $(\epsilon_1, \epsilon_2)$ \\ \hline
$(0,4)$ &$\epsilon_1 \equiv \epsilon_2 \equiv 1~ (4)$ & $(5,6)$ &$2\epsilon_1+\epsilon_2 \in Q(\langle1,2\epsilon_1 \rangle)$ \\

$(1,6)$ &$\epsilon_2 \in Q(\langle 1, 2\epsilon_1 \rangle)$ & $(5,7)$ & $\epsilon_1\epsilon_2\equiv1~(4)$\\

$(2,2)$ &$\epsilon_1=1,~\epsilon_2\equiv3~(4)$\quad & $(5,8)$ &$\epsilon_2\equiv2\epsilon_1+5~(8)$ \\

$(2,4)$ & $\epsilon_1\equiv1~(4)$& $(5,9)$ & $\epsilon_1\epsilon_2\equiv 1~(4)$\\

$(2,6)$ & $\epsilon_1\equiv1~(4)$ & $(5,2k),~(k\geq 5)$ &$1+2\epsilon_1\nequiv \epsilon_2~(8)$ \\

$(2,2k-1),~(k\geq 4)$ &$\e_1\equiv2\e_2+3~(8)$& 
$(5,2k+1),~(k\geq 5)$ & $1+2\epsilon_1\nequiv \epsilon_1\epsilon_2~(8)$ \\

$(2,2k),~(k\geq 4)$ & $\epsilon_1\equiv1~(4)$ & $(6,7)$ & $5\notin Q(\langle\epsilon_1,2\epsilon_2\rangle)$ \\

$(3,6)$ & $\epsilon_2\equiv1~(8)$ & $(6,9)$ & $5\notin Q(\langle\epsilon_1,2\epsilon_2\rangle)$\\

$(4,4)$ & $\epsilon_1\equiv \epsilon_2\equiv1~(4)$& $(6,2k-1),~(k\geq 6)$ &$\epsilon_1\nequiv5~(8)$ \\

$(5,5)$ & $\epsilon_2\equiv3\epsilon_1+6~(8)$ & $(6,2k),~(k\geq6)$ & $\epsilon_1,\epsilon_2\nequiv5~(8)$ and\\ 
&&&$\e_1\nequiv \e_2 \Rightarrow \e_1 ~\text{or}~\e_2 \equiv 1 ~(8)$\\
\hline
\end{tabular}
\end{footnotesize}
\end{center}

\begin{lem} \label{H-type rmk}
Let $p$ be a prime and let $L$ be a primitive ternary $\z$-lattice.  As a function from $\gen(L)_S$ to $\gen(\lambda_p(L))_S$, $\lambda_p$ is 
a $2^a$ to one function for some $a=0,1$ or $2$. 
\end{lem}

\begin{proof}  
Note that if $L$ is not of $H$-type at $p$,  then
$$
|\theta(O^+(\lambda_p(L_p)))|=\begin{cases} 4\cdot |\theta(O^+(L_p))| &\text{if}~p=2,~ (\alpha,\beta)=(2,4)  \ \text{and}~\e_1\equiv\e_2\equiv1~(4),\\
2\cdot|\theta(O^+(L_p))| &\text{otherwise}.\end{cases}
$$
 Suppose that $\lambda_p(\spn(L))=\spn(M)$ with $\lambda_p(L)=M$. For any $\spn(M')\in \gen(M)_S$,  there is a split rotation $\Sigma\in J_V$ such that $M'=\Sigma M$. Since
$$
\lambda_p(\Sigma L)=\Sigma\lambda_p(L)=\Sigma M=M',
$$
we have $\lambda_p(\spn(\Sigma L))=\spn(M')$.
Note that $L'\in\spn(L'')$ if and only if $\Sigma L'\in\spn(\Sigma L'')$, for any $L',L''\in\gen(L)$. 
Therefore $\vert\lambda_p^{-1}(\spn(M))\vert$ is independent of the choices of $M \in \gen(\lambda_p(L))$. 
The lemma follows from this and the fact that $\lambda_p$ is surjective and the number of spinor genera in any genus of a ternary quadratic form is a power of $2$.  \end{proof}

\section{$\Gamma_p$-transformations on the set of spinor genera}

Let $V$ be a (positive definite) ternary quadratic space and
let $L$ be a primitive ternary $\z$-lattice on $V$. Let $p$ be a prime such that  
$$
L_p \simeq \begin{pmatrix} 0&\frac12\\ \frac12&0\end{pmatrix} \perp \langle \epsilon \rangle,~ \text{where}~ \epsilon \in \z_p^{\times}.
$$ 
For any nonnegative integer $m$, let $\mathcal G_{L,p}(m)$ be a genus on $W$ such that each $\z$-lattice $M \in \mathcal G_{L,p}(m)$ satisfies
$$
 M_p \simeq  \begin{pmatrix} 0&\frac12\\ \frac12&0\end{pmatrix} \perp \langle \epsilon p^m \rangle \quad \text{and} \quad M_q \simeq (L^{p^m})_q \ \text{ for any $q \ne p$}.
$$
Here $W=V$ if $m$ is even, and $W=V^p$ otherwise. 

For a nonnegative integer $m$, let $N \in \mathcal{G}_{L,p}(m+1)$ be a primitive ternary $\z$-lattice. 
By Weak Approximation Theorem, there exists a basis $\{x_1, x_2, x_3 \}$ for $N$ such that
$$
(B(x_i,x_j))\equiv\begin{pmatrix}0&\frac12\\ \frac12&0\end{pmatrix}\perp \langle p^{m+1} \delta\rangle \ (\text{mod} \ p^{m+2}),
$$
where $\delta$ is an integer not divisible by $p$. We define two sublattices of $N$
$$
\aligned
&\Gamma_{p,1}(N) = \z px_1 + \z x_2+ \z x_3, &\Gamma_{p,2}(N) = \z x_1 + \z px_2+ \z x_3.
\endaligned
$$
Note that $\Gamma_{p,i}(N)$ for $i=1,2$ depends on the choices of basis for $N$. However, the set $\{\Gamma_{p,1}(N), \Gamma_{p,2}(N)\}$ of sublattices of $N$ is independent of the choices of basis for $N$. In fact, these two sublattices are unique sublattices of $N$ with index $p$ whose norm is $p\z$. We say that a $\z$-lattice $M$ is {\it a $\Gamma_p$-descendant of $N$} if $M \simeq \Gamma_{p,i}^{\frac1p}(N)$ for some $i=1,2$.

\begin{lem}\label{exchange} Let $p,q$ be distinct  primes and let $N\in\mathcal{G}_{L,p}(m+1)$ for some nonnegative integer $m$. 
\begin{enumerate}
\item  If $M$ is a $\Gamma_p$-descendant of $N$, then $\lambda_q(M)$ is a $\Gamma_p$-descendant of $\lambda_q(N)$.

\item  Assume that $N\in \mathcal{G}_{L',q}(m'+1)$ for some nonnegative integer $m'$. Then any $\Gamma_q$-descendant of a $\Gamma_p$-descendant of $N$ is  a $\Gamma_p$-descendant of some $\Gamma_q$-descendant of $N$.
\end{enumerate}
\end{lem}

\begin{proof} Note that if $p,q$ are distinct primes, then $(\Gamma_{p,i}(N))_q=N_q$ and $(\Lambda_p(N))_q=N_q$.     
The lemma follows directly from this. \end{proof}

In \cite{graph}, we defined a multi-graph $\mathfrak{G}_{L,p}(m)$ and proved some properties of this graph.  
For those who are unfamiliar with the notations, we introduce the definition of this multi-graph briefly: the set of vertices in $\gr$ is the set of equivalence classes in  $\mathcal G_{L,p}(m)$, say, $\{[M_1], [M_2], \ldots , [M_h] \}$. The set of edges is exactly the set of equivalence classes in $\mathcal G_{L,p}(m+1)$, say, $ \{[N_1], [N_2],\ldots,[N_k] \}$. For each equivalence class $[N_w] \in\mathcal G_{L,p}(m+1)$, two vertices contained in the edge $[N_w]$ are defined by $[\Gamma_{p,1}(N_w)^{\frac 1p}]$ and  $[\Gamma_{p,2}(N_w)^{\frac1p}]$ that are defined above. Note that both lattices are contained in $\mathcal G_{L,p}(m)$. Hence this graph might have loops or multiple edges.

 Two vertices $[T_i], [T_j] \in \mathfrak{G}_{L,p}(0)$ are connected by an edge if and only if there are $\z$-lattices $T_i' \in [T_i]$ and $T_j' \in [T_j]$ such that $T_i'$ and $T_j'$ are connected by an edge in the graph $Z(T,p)$ which is defined in \cite{sp1}. If two lattices $T_i ,T_j \in \mathcal G_{L,p}(0)$ are spinor equivalent, then  both $[T_i]$ and $[T_j]$ are contained in the same connected component. Moreover, the set of vertices in each connected component of $\mathfrak{G}_{L,p}(0)$ consists of  at most two spinor genera, and it consists of  only one spinor genus if and only if $\bold{j}(p) \in P_D J_{\q}^K$, where $D$ is the set of positive rational numbers and 
$$
\bold{j}(p) = (j_{q}) \in J_\q \quad \text{such that $j_p =p$ and $j_q = 1$ for any prime $q \ne p$}.
$$
 We say that $\mathfrak{G}_{L,p}(0)$ is of $O$-type if the set of vertices in the connected component of the graph $\mathfrak{G}_{L,p}(0)$ consists of  only one spinor genus, and it is of $E$-type otherwise.

Now, we consider the general case. 
For any positive integer $m$, we say that a graph $\mathfrak G_{L,p}(m)$ is of $E$-type if $m$ is even and 
$\mathfrak G_{L,p}(0)$ is of $E$-type, and it is of $O$-type otherwise.  

Assume that  $\mathfrak G_{L,p}(m)$ is of $E$-type and $M \in \mathcal G_{L,p}(m)$. Since the map 
$$
\lambda_p^{\frac m2} : \spn(T) \to \spn(\lambda_p^{\frac m2}(T))
$$
 is surjective for any $T \in \mathcal G_{L,p}(m)$, 
there is a $\z$-lattice $M' \in \mathcal G_{L,p}(m)$ such that $M' \not \in \spn(M)$ and $[M']$ is connected to $[M]$ by a path by Lemma 3.5 in \cite{graph}.
Note that $g(\mathcal G_{L,p}(m))=g(\mathcal G_{L,p}(m'))$ if and only if $m\equiv m' \pmod 2,$ where $g(\mathcal G_{L,p}(m))$ is the number of spinor genera contained in $\mathcal G_{L,p}(m)$.
In particular,
$g(\mathcal G_{L,p}(m))=g(\mathcal G_{L,p}(0))$ for any even $m$. So, every $\z$-lattice $M'$ satisfying the above condition is contained in a single spinor genus.  
 From the existence of such a $\z$-lattice $M'$, we may define  
$$
\text{Cspn}(M)=\begin{cases} \text{spn}(M) \quad &\text{if $\mathfrak G_{L,p}(m)$ is of $O$-type,}\\
 \text{spn}(M) \cup \spn(M')  \quad &\text{otherwise}.
 \end{cases}
$$
In Lemma 3.10 of \cite{graph}, we proved that the set of all vertices in the connected component of $\mathfrak G_{L,p}(m)$ containing $[M]$ is the set of equivalence classes in $\text{Cspn}(M)$. 

\begin{lem} \label{coco} For an integer $m\geq0$, let $[N] \in \mathcal G_{L,p}(m+1)$ be an edge of the graph $\mathfrak G_{L,p}(m)$.  Then the set of all edges in the connected component of $\mathfrak G_{L,p}(m)$ containing $[N]$  is the set of all classes in $\text{Cspn}(N)$.
\end{lem}

\begin{proof}  It suffices to show that the set of edges in the connected component of $\mathfrak G_{L,p}(m)$ containing $[N]$
is exactly the set of vertices in the connected component of $\mathfrak G_{L,p}(m+1)$ containing the vertex $[N]$ by Lemma 3.10 of \cite{graph}.  Note that if $N_1$ and $N_2$ are different $\Gamma_p$-descendant of $K$ for some $K \in \mathcal G_{L,p}(m+2)$, then $\lambda_p(K)$ is a $\Gamma_p$-descendant of both $N_1$ and $N_2$. This implies that every class in $\text{Cspn}(N)$ is contained in the set of edges  in the connected component of $\mathfrak G_{L,p}(m)$ containing $[N]$. Conversely, assume that $[N']$ is   contained in the set of edges  in the connected component of $\mathfrak G_{L,p}(m)$ containing $[N]$.  Without loss of generality, we may assume that there is a $\z$-lattice $M$ that is a $\Gamma_p$-descendant of both $N$ and $N'$. If $m=0$ or $m\ge 1$ and $\lambda_p(N) \ne \lambda_p(N')$, then there is a $\z$-lattice $K$ whose $\Gamma_p$-descendants are both $N$ and $N'$ by Lemmas 3.2 and 3.3 of \cite{graph}, that is, as vertices, $[N]$ and $[N']$ are contained in the edge $[K]$. Now suppose that $m\ge1$ and  $\lambda_p(N)= \lambda_p(N')$. Then in this case, there is a $\z$-lattice $S \in \mathcal G_{L,p}(m+1)$  such that $\lambda_p(S)\ne \lambda_p(N)$ and one of $\Gamma_p$-descendants of $S$ is $M$ (see Lemma 3.3 of \cite{graph}). Hence there are edges containing $\{[N], [S]\}$ and $\{[S], [N']\}$  in the graph $\mathfrak G_{L,p}(m+1)$. This completes the proof.       
\end{proof}

\section{genus-correspondences}

Let $n$ be a positive integer.
Let $M$ be a ternary $\mathbb{Z}$-lattice on a quadratic space $V$ and let $N$ be a $\mathbb{Z}$-lattice on $V^n$. Assume that there is a representation 
$$
\phi : M^n \to N  \ \  \text{such that} \ \ [N:\phi(M^n)]=n.
$$ 
Then clearly, $N^n$ is also represented by $M$. For any $\z$-lattice $M_1 \in \gen(M)$, since $(M_1^n)_p \simeq (M^n)_p \ra N_p$ for any prime $p$, there is a $\z$-lattice $N_1 \in \gen(N)$ that represents $(M_1)^n$. Conversely, for any $\z$-lattice $N' \in \gen(N)$, there is a $\z$-lattice $M' \in \gen(M)$ such that $(M')^n \ra N'$ (see \cite{Jagy}). For $M_1 \in \gen(M)$ and $N_1 \in \gen(N)$ such that $(M_1)^n \ra N_1$,  
 the pair $([N_1],[M_1]) \in \gen(N)/\sim\times \gen(M)/\sim$ is called {\it a representable pair by scaling $n$}. A subset $\mathfrak C \subset  \gen(N)/\sim\times \gen(M)/\sim$ consisting of representable pairs by scaling $n$ is called {\it a genus-correspondence} if  
for any $N' \in \gen(N)$, there is an $M' \in \gen(M)$ such that $([N'],[M']) \in \mathfrak C$, and vice versa. We say a genus-correspondence $\mathfrak C$ {\it respects spinor genus} if for any two $([N_1],[M_1]), ([N_2],[M_2]) \in \mathfrak C$,
$$
N_1 \in \spn(N_2) \qquad  \text{if and only if} \qquad  M_1 \in \spn(M_2).
$$ 
 Concerning this, Jagy conjectured in \cite{Jagy} that if $n$ is square free and $g(N)=g(M)$, then any genus-correspondence respects spinor genus. However, the following example shows that the conjecture is not true. 

\begin{exam}\label{exam3} {\rm Let $N_1=\langle 12 \rangle \perp \begin{pmatrix} 15&5\\5&135 \end{pmatrix}$ and $M_1=\langle 1,~20,~80 \rangle$. Then one may easily check that $g(M_1)=g(N_1)=2$,  $dN_1 = 15\cdot dM_1$ and $M_1^{15}$ is represented by $N_1$. The genus of $N_1$ consists of the following $12$ lattices up to isometry:
$$
\begin{array}{llll}
\!\!\!\!\! &N_1=\langle 12,15,135,5,0,0\rangle,\!\!\!\! &N_2=\langle3,7,1200,0,0,1 \rangle,\! \!\!&N_3= \langle 3,60,140,20,0,0 \rangle, \\     
\!\!\!\!\! &N_4=\langle3,27,300,0,0,1 \rangle, \!\! \!\!&N_5=\langle 27,27,40,10,10,3\rangle,\! \!\!&N_6=\langle12,28,83,12,4,-4 \rangle,\\ 
\!\!\!\!\!   &N_7=\langle12,28,75,0,0,4 \rangle, \!\! \!\! &N_8=\langle15,35,48,0,0,5 \rangle, \!\!\! &N_9=\langle7,12,300,0,0,2 \rangle, \\ 
\!\!\! \!\! &N_{10}=\langle12,43,60,20,0,6 \rangle, \!\! \! \! &N_{11}=\langle 8,12,303,4,-2,4 \rangle,\! \!\! &N_{12}=\langle12,35,60,10,0,0 \rangle.
\end{array}
$$
Note that up to isometry,
$$
\text{spn}(N_1)=\{N_i : 1\le i\le6\}\quad \text{and}  \quad \text{spn}(N_7)=\{ N_i : 7\le i\le12\}.
$$ 
 The genus of $M_1$ consists of the following $6$ lattices up to isometry:
$$
\begin{array}{llll}
&\hspace{-5mm}M_1=\langle1,20,80,0,0,0 \rangle, &M_2=\langle5,16,20,0,0,0 \rangle, &M_3=\langle4,20,25,10,0,0 \rangle, \\
&\hspace{-5mm}M_4=\langle4,5,80,0,0,0 \rangle, &M_5=\langle9,9,20,0,0,1 \rangle, &M_6=\langle4,20,21,0,2,0 \rangle.
\end{array}
$$
Note that up to isometry,
$$
\text{spn}(M_1)=\{M_i : 1\le i\le3\} \quad \text{and}\quad  \text{spn}(M_4)=\{ M_i : 4\le i\le6\}.
$$ 
Define a genus-correspondence $\mathfrak{S}$ as follows:
$$
\aligned
\mathfrak{S}=\{&([N_1],[M_1]),([N_9],[M_1]),([N_3],[M_2]),([N_7],[M_2]),\\
&([N_5],[M_3]),([N_{11}],[M_3]),([N_2],[M_4]),([N_8],[M_4]),\\
&([N_6],[M_5]),([N_{10}],[M_5]),([N_4],[M_6]),([N_{12}],[M_6])\}.
\endaligned
$$
Then one may easily check that $\mathfrak{S}$ does not respect spinor genus.}
\end{exam}

In the remaining,  we show that if we take a genus-correspondence suitably, then it respects spinor genus under the assumption that $g(M)=g(N)$. We do not assume that $n$ is square free for a while.

\begin{lem}\label{spn representable}
For ternary $\mathbb{Z}$-lattices $N$ and $M$, assume that $([N],[M])$ is a representable pair by scaling $n$. 
Then for any $N'\in\text{spn}(N)$, there is a $\mathbb{Z}$-lattice $M'\in\text{spn}(M)$ such that $([N'],[M'])$ is a representable pair by scaling $n$. 
Conversely, for any $M''\in\text{spn}(M)$ there is a $\mathbb{Z}$-lattice $N''\in\text{spn}(N)$ such that $([N''],[M''])$ is a representable pair by scaling $n$. 
\end{lem}

\begin{proof}
Since $([N],[M])$ is a representable pair by scaling $n$, $\sigma(M^n)\subseteq N$, for some isometry $\sigma\in O(V)$.
Let $N'\in\text{spn}(N)$. Then there are $\sigma'\in O(V)$ and $\Sigma\in J'_V$ such that $N'=\sigma'\Sigma N$.
If we define $M'=\sigma'\Sigma\sigma(M)=\sigma'\sigma(\sigma^{-1}\Sigma\sigma) M \in \text{spn}(M)$, then 
$$
(M')^n= \sigma'\Sigma\sigma (M^n) \subseteq \sigma'\Sigma N=N'.
$$
The converse can be proved similarly.
\end{proof}

A bipartite graph  with partitions $U$ and $V$ of vertices and  with $E$ of edges is   denoted by $\mathfrak G(U,V,E)$, or simply $\mathfrak G(U,V)$. 
For each vertex $u \in U$ of the bipartite graph $\mathfrak G(U,V,E)$,  we define $\mathcal N(u)=\{ v : uv \in E\}$. For a vertex $v \in V$, $\mathcal N(v)$ is defined similarly. 
The graph $\mathfrak G(U,V,E)$ is called {\it $(a,b)$-regular} if $\mathcal N(u)=a$ for any $u \in U$, and $\mathcal N(v)=b$ for any $v \in V$.  

 For two bipartite graphs
$\mathfrak G(U,V,E)$ and $\mathfrak G(V,W,E')$, we define   a {\it juxtaposition bipartite graph} of two bipartite graphs, denoted by $\mathfrak G_V(U,W,\tilde E)$,
as follows;  $U$ and $W$ are partitions of vertices and 
 there is an edge $uw \in \tilde E$ for $u \in U$ and $w \in W$
if and only if there is a vertex $v \in V$ such that $uv \in E$ and $vw \in E'$.     

For a representable pair $([N],[M])$ by scaling $n$, we define a bipartite graph 
$$
\mathfrak{G}(N,M)=\mathfrak G(\gen(N)_S,\gen(M)_S)
$$
 such that 
 two vertices $\text{spn}(N')\in\text{gen}(N)_S$ and $\text{spn}(M')\in\text{gen}(M)_S$ are connected by an edge if and only if there are lattices $N''\in \text{spn}(N')$ and $M''\in\text{spn}(M')$ such that $([N''],[M''])$ is a representable pair by scaling $n$.

\begin{lem}\label{regular}
Let $N$ and $M$ be two $\z$-lattices such that $([N],[M])$ is a representable pair by scaling $n$. Then for some positive integers $u,v$ such that $ug(N)=vg(M)$, the graph $\mathfrak G(N,M)$ is $(u,v)$-regular.  In particular,  if $g(N)=g(M)$, then the graph $\mathfrak{G}(N,M)$ is a regular bipartite graph.
\end{lem}

\begin{proof}
Let $\text{spn}(N')\in\text{gen}(N)_S$ and $\text{spn}(M')\in\text{gen}(M)_S$ be two vertices the graph $\mathfrak{G}(N,M)$ such that $\text{spn}(M')\in \mathcal{N}(\text{spn}(N'))$. 
By Lemma \ref{spn representable}, we may assume that $(N',M')$ is a representable pair by scaling $n$, that is, there is a representation $\phi \in O(V)$ such that $\phi((M')^n)\subset N'$.  
Let $\text{spn}(N'')$ be another vertex in $\text{gen}(N)_S$. 
Choose a split rotation $\Sigma\in J_V$ such that $N''=\Sigma N'$.  Since 
$$
\phi(\phi^{-1}\Sigma\phi\Sigma^{-1}(\Sigma(M')^n))\subset \Sigma(N'),
$$
$(\Sigma N',\phi^{-1}\Sigma\phi\Sigma^{-1}(\Sigma(M')))$ is a representable pair by scaling $n$.  Furthermore, since $\phi^{-1}\Sigma\phi\Sigma^{-1} \in J_V'$, we have $\text{spn}(\Sigma M') \in \mathcal{N}(\text{spn}(N''))$.
Note that for any two lattices $M',M''\in \text{gen}(M)$, 
$M'\in\text{spn}(M'')$ if and only if $\Sigma M'\in\text{spn}(\Sigma M'')$.
Therefore 
$$
| \mathcal{N}(\text{spn}(N'))|= |\mathcal{N}(\text{spn} (N''))|.
$$ 
Similarly, we also have $| \mathcal{N}(\text{spn}(M'))|= |\mathcal{N}(\text{spn} (M''))|$ for any $M', M'' \in \gen(M)$. The lemma follows from this. \end{proof}

\begin{thm}\label{exist}
Let $N$ and $M$ be two $\z$-lattices such that $([N],[M])$ is a representable pair by scaling $n$. If $g(N)=g(M)$, then there is a genus-correspondence respecting spinor genus.
\end{thm}

\begin{proof} We may assume that 
$$
\gen(N)_S =\{ \spn(N_i) : i=1,2,\dots,g\}\ \ \text{and}   \ \ \gen(M)_S =\{ \spn(M_i) : i=1,2,\dots,g\}.
$$
Since the graph $\mathfrak{G}(N,M)$ defined above is a regular bipartite graph, there is a perfect matching by Hall's marriage theorem.  Hence, without loss of generality, we may assume that each $([N_i],[M_i])$ is a representable pair by scaling $n$.  
We define a genus-correspondence $\mathfrak{S}$ as follows:
 for $([N'],[M'])\in\text{gen}(N)/\sim\times\text{gen}(M)/\sim$, $([N'],[M'])\in\mathfrak{S}$ if and only if $([N'],[M'])$ is a representable pair by scaling $n$ and there is an $i \ (1\le i\le g)$ such that $N'\in \spn(N_i)$ and $M' \in \spn(M_i)$. 
 Then by Lemma \ref{spn representable}, $\mathfrak{S}$ is a genus-correspondence respecting spinor genus.
\end{proof}

\section{Genus-correspondence respecting spinor genus}

From now on, we assume that $n$ is a square free  positive integer. Let $N$ and $M$ be ternary $\z$-lattices such that the pair  $([N], [M])$ is a representable pair by scaling $n$. 
 In this section, We find a necessary and sufficient condition for the genus-correspondence  $\gen(N)/\sim\times\gen(M)/\sim$ to respect spinor genus under the assumption that $g(M)=g(N)$. 
 
\begin{lem} \label{important}  Let $p$ be a prime and let $N$ and $M$  be primitive ternary $\z$-lattices.  Then $([N],[M])$  is a representable pair by scaling $p$ if and only if 
$M^p \simeq \Lambda_p(N)$  or $M$ is a $\Gamma_p$-descendant of $N$.  
\end{lem}  

\begin{proof} Without loss of generality, we may assume that $M^p$ is a sublattice of $N$ with index $p$. Then there is  a basis $\{x_1,x_2,x_3\}$ for $N$ such that $M^p=\z x_1+\z x_2+\z (px_3)$. Hence there are integers $a,b,c,s,t,u$ such that  
$$
(B(x_i,x_j))=\begin{pmatrix}pa&\frac{pu}2&\frac t2\\\frac{pu}2&pb&\frac s2\\\frac t2&\frac s2&c\end{pmatrix}.
$$   
Hence if $s\equiv t\equiv 0\pmod p$, then clearly, $M^p=\Lambda_p(N)$. If $s$ or $t$ is not divisible by $p$, then a Jordan decomposition of $N_p$ has an isotropic $\frac 12\z_p$-modular component.  Furthermore, $M^p$ is a sublattice of $N$ with index $p$ whose norm is $p\z$. Therefore $M$ is a $\Gamma_p$-descendant of $N$.    Note that the converse is almost trivial. \end{proof}

\begin{lem}  \label{step-r} For two ternary $\z$-lattices $N$ and $M$, assume that $([N],[M])$  is a representable pair by scaling $n$. For any prime $p$ dividing $n$, there is a $\z$-lattice $N(p)$ such that $([N],[N(p)])$ is a representable pair by scaling $p$, and $([N(p)],[M])$ is a representable pair by scaling $\frac np$. 
\end{lem}

\begin{proof} By assumption, we may assume that $M^n$ is a sublattice of $N$ with index $n$. Choose a basis $\{x_1,x_2,x_3\}$ for $N$ such that $M^n=\z x_1+\z x_2+\z (nx_3)$. Define $N(p)=(\z x_1+\z x_2+\z (px_3))^{\frac1p}$. Then one may easily show that $\mathfrak n(N(p))=\z$ and $N(p)$ satisfies all conditions given above. \end{proof}

\begin{cor}  \label{step-r2} Let $a,b$ be positive integers such that $ab$ is square free. Let $([N],[L])$ and $([L],[M])$ be representable pairs of ternary $\z$-lattices by scaling  $a$ and $b$, respectively. Then the graph $\mathfrak G(N,M)$ is  exactly same to the juxtaposition bipartite graph $\mathfrak G_{\gen(L)_S}(\gen(N)_S,\gen(M)_S)$. 
\end{cor} 

\begin{proof}  It suffices to show that each set of edges for both graphs is same, which follows directly  from the above lemma. 
\end{proof}

Let $([N],[M])$ be a representable pair of ternary $\z$-lattices by scaling $p$, where $p$ is a prime. Then $\lambda_p(N) \simeq M$ or $M$ is a $\Gamma_p$-descendent of $N$ by Lemma \ref{important}. If the former holds, then the bipartite graph $\mathfrak G(N,M)$ is $(1,1)$-regular or $(1,2)$-regular by  Lemma \ref{H-type rmk}.  Furthermore, it is $(1,1)$-regular if and only if $N$ is of $H$-type at $p$. Note that $(1,4)$-regularity is impossible in our situation. 

Now, assume that the latter holds.  Then there is a $\z$-lattice $L$ and a nonnegative integer $m$ such that $N \in \mathcal G_{L,p}(m+1)$ and $M \in \mathcal G_{L,p}(m)$.   
If the graph $\mathfrak G_{L,p}(m+1)$ is of $E$-type, then the bipartite graph $\mathfrak G(N,M)$ is $(1,2)$-regular by Lemma \ref{coco}, and if  the graph $\mathfrak G_{L,p}(m)$ is of $E$-type, then the bipartite graph $\mathfrak G(N,M)$ is $(2,1)$-regular. Finally, if both $\mathfrak G_{L,p}(m+1)$ and $\mathfrak G_{L,p}(m)$ are of $O$-type, then the bipartite graph $\mathfrak G(N,M)$ is $(1,1)$-regular. Note that  both $\mathfrak G_{L,p}(m+1)$ and $\mathfrak G_{L,p}(m)$ cannot be of $E$-type  simultaneously.  We say  $N$ is of $(E,O)$-type ($(O,E)$-type) at $p$ if  the graph $\mathfrak G_{L,p}(m+1)$ ($\mathfrak G_{L,p}(m)$, respectively) is of $E$-type.  Finally, we say $N$ is of $(O,O)$-type if both  $\mathfrak G_{L,p}(m+1)$ and $\mathfrak G_{L,p}(m)$ are of $O$-type.

Let $([N],[M])$ be a representable pair of ternary $\z$-lattices by scaling $n$, where $n$ is a square free positive integer.
Without loss of generality, we assume that $M^n$ is a sublattice of $N$ with index $n$.  Choose a basis $\{x_1,x_2,x_3\}$ for $N$ such that $M^n=\z x_1+\z x_2+\z (nx_3)$. Let $n_2$ be a product of primes $p$ dividing $n$ such that  the rank of a $\frac12\z_p$-modular component in a Jordan decomposition of $N_p$ is two, and let $n_1$ be the integer satisfying $n=n_1n_2$.  Let $n_2(e)$ be a product of primes $q$ dividing $n_2$ such that $\ord_q(4\cdot dM) \equiv 0 \pmod 2$, and let $n_2(o)$  be the integer satisfying $n_2=n_2(e)n_2(o)$.  Define a ternary $\z$-lattice
$$
L_{N,M}=(\z x_1+\z x_2+\z (n_1n_2(e)x_3))^{\frac1{n_1n_2(e)}}.
$$
 Note that pair $([N],[L_{N.M}])$ ($([L_{N,M}],[M])$) is a representable pair by scaling $n_1n_2(e)$  ($n_2(o)$, respectively).   
Let $g_{N,M}=g(L_{N,M})$ be the number of (proper) spinor genera in the genus of $L_{N,M}$. 

\begin{thm}  \label{resol} Let $n$ be a square free positive integer and let $([N],[M])$ be a representable pair by  scaling $n$. Then, any connected component of the bipartite graph $\mathfrak G(N,M)$ is a complete $K_{\alpha,\beta}$-graph, where 
$$
\alpha=\frac{g(M)}{g_{N,M}} \qquad \text{and} \qquad \beta=\frac{g(N)}{g_{N,M}}.
$$   
\end{thm} 

\begin{proof}   Without loss of generality, we assume that $M^n$ is a sublattice of $N$ with index $n$.
 Let $n_1n_2(e)=p_1p_2\dots p_s$ and $n_2(o)=q_1q_2\dots q_t$, where each $p_i$ and $q_j$ is a prime.    By Lemma \ref{step-r2}, the graph $\mathfrak G(N,M)$ is a juxtaposition of the graphs $\mathfrak G(N,L_{N,M})$ and $\mathfrak G(L_{N,M},M)$, where $L_{N,M}$ is a $\z$-lattice defined above.  Let $\{x_1,x_2,x_3\}$ be a basis for $N$ such that 
$L_{N,M}^{n_1n_2(e)}=\z x_1+\z x_2+\z (n_1n_2(e)x_3)$. Define 
$$
N(i)=(\z x_1+\z x_2+\z (p_1p_2\dots p_ix_3))^{\frac1{p_1p_2\dots p_i}}.
$$ 
Then the graph $\mathfrak G(N,L_{N,M})$ is a juxtaposition of the graphs 
$$
\mathfrak G(N,N(1)), \mathfrak G(N(1),N(2)), \dots, \mathfrak G(N(s-1),L_{N,M}),
$$
all of which are either a $(1,1)$-regular graph or a $(1,2)$-regular graph. Therefore the graph $\mathfrak G(N,L_{N,M})$ is a $\left(1,\frac {g(N)}{g_{N,M}}\right)$-regular graph.  Similarly, one may easily check that the graph $\mathfrak G(L_{N,M},M)$ is a $\left(\frac{g(M)}{g_{N,M}},1\right)$-regular graph.  The theorem follows from these two observations. 
\end{proof}

\begin{cor}  \label{resol2} Let $n$ be a square free positive integer and let $([N],[M])$ be a representable pair by  scaling $n$. Assume that $g(N)=g(M)$. Then $g(N)=g_{N,M}$ if and only if the genus-correspondence  $\gen(N)/\sim\times\gen(M)/\sim$ respects spinor genus. 
\end{cor}

\begin{proof}
Note that the genus-correspondence  
$\gen(N)/\sim\times\gen(M)/\sim$ respects spinor genus if and only if the graph $\mathfrak G(N,M)$ is $(1,1)$-regular.  Hence the corollary follows directly from the above theorem. 
\end{proof}

Recall that we are assuming that $n$ is a square free positive integer and $([N],[M])$ is a representable pair by scaling $n$. Now we further assume that $g(N)=g_{N,M}=g(M)$, that is,  the genus-correspondence  $\gen(N)/\sim\times\gen(M)/\sim$ respects spinor genus by Theorem \ref{resol}.  Assume that
$$
\gen(N)_S=\{\spn(N_i):i=1,2,\dots,g\} \ \ \text{and}\ \ \gen(M)_S=\{\spn(M_i):i=1,2,\dots,g\},
$$  
 and there is a unique edge containing $\spn(N_i)$ and $\spn(M_i)$ in the graph  $\mathfrak{G}(N,M)$ for any $i=1,2,\dots,g$.  

\begin{lem}  Let $k$ be a positive integer.
Under the assumptions given above, $\spn(N_i)$ represents $nk$ if and only if $\spn(M_i)$ represents $k$, for any $i=1,2,\dots,g$. 
\end{lem}

\begin{proof} First assume that $n=p$ is a prime.  Then  $\lambda_p(N)\simeq M$ or $M$ is a $\Gamma_p$-descendant of $N$ by Lemma \ref{important}. If the former holds, one may easily show that $N$ represents $pk$ if and only if $M$ represents $k$.
If the latter holds, then also one may easily show that $N$ represents $pk$ if and only if at least one of $\Gamma_p$-descendants of $N$ represents $k$. Hence the lemma follows directly from this.

Let $n=p_1p_2\cdots p_r$, where each $p_i$ is a prime. We may assume that $M^n$ is a sublattice of $N$ with index $n$. Let $\{x_1,x_2,x_3\}$ be a basis for $N$ such that $M^n=\z x_1+\z x_2+ \z (nx_3)$. Define 
$$
N(i)=(\z x_1+\z x_2+\z (p_1p_2\dots p_ix_3))^{\frac1{p_1p_2\dots p_i}}.
$$ 
Then the graph $\mathfrak G(N,M)$ is a juxtaposition of the graphs 
$$
\mathfrak G(N,N(1)), \mathfrak G(N(1),N(2)), \dots,  \mathfrak G(N(r-1),M).
$$
Since $g(N)=g_{N,M}=g(M)$, we know that each graph $\mathfrak G(N(i),N(i+1))$ is a $(1,1)$-regular graph. Hence the lemma follows directly from the induction on $r$.
\end{proof}

A set $S=\{c_1,c_2,\dots,c_g\}$ of integers is said to be a complete system of spinor exceptional integers for $\gen(L)$, for some ternary $\z$-lattice $L$, if for any subset $U \subset  S$, there is a unique $\spn(L') \in \gen(L)_S$ such that every integer in $U$ is represented by $\spn(L')$ and every integer in $S-U$ is not represented by $\spn(L')$. For details, see \cite {bh}.

\begin{cor}\label{complete sys}
Under the same assumptions given above, suppose that there is a complete system $\{c_1,c_2,\dots,c_g\}$ of spinor exceptional integers for $\gen(M)$. Then $\{nc_1,$ $nc_2, \cdots,nc_g\}$ is a complete system of spinor exceptional integers for $\gen(N)$.  
\end{cor}

\begin{proof}
The corollary follows directly form the above lemma.
\end{proof}

The following example was first introduced by Jagy in \cite{Jagy}.

\begin{exam} {\rm
Let $n$ be a square free integer that is represented by a sum of two integral squares. Define ternary $\z$-lattices 
$$
N=\langle1,1,16n\rangle \ \ \text{and}\ \ M=\langle1,1,16\rangle.
$$ 
Then one may easily check that $([N],[M])$ is a representable pair by scaling $n$ and $g(N)=g_{N,M}=g(M)=2$.
Note that $\{1\}$ is a complete system of spinor exceptional integers for $\gen(M)$. By Corollary \ref{complete sys}, we know that $\{n\}$ is a complete system of spinor exceptional integers for $\gen(N)$.
In fact, Jagy proved in \cite{Jagy} that $n$ is represented by any $\z$-lattice in $\spn(N)$. However one may easily show that it holds even if  $n$ is not square free.}
\end{exam}

\end{document}